\documentclass[reqno,A4paper]{amsart}

\usepackage{amsmath}
\usepackage{amssymb}
\usepackage{amsthm}
\usepackage{enumerate}
\usepackage{enumitem} 
\usepackage{mathrsfs}
\usepackage{eqlist}
\usepackage{array}
\usepackage{float}
\usepackage{hyperref}
\usepackage{tikz}
\usepackage{forest}

\usepackage{tikz}
\usepackage{tikz-qtree}

\theoremstyle{plain}
\newtheorem{theorem}{Theorem}[section]

\newtheorem{lemma}[theorem]{Lemma}
\newtheorem{corol}[theorem]{Corollary}

\theoremstyle{definition}

\numberwithin{equation}{section}

\begin{document}

\title[Markoff $m$-triples with $k$-Fibonacci components]{Markoff $m$-triples with $k$-Fibonacci components}

\author[D. Alfaya, L. A. Calvo, A. Mart\'inez de Guinea, J. Rodrigo \and A. Srinivasan]{D. Alfaya*  ***, L. A. Calvo**, A. Mart\'inez de Guinea***, J. Rodrigo* \and A. Srinivasan**}

\newcommand{\acr}{\newline\indent}

\address{\llap{*\,}Department of Applied Mathematics,\acr
ICAI School of Engineering, Comillas Pontifical University,\acr
C/Alberto Aguilera 25, 28015 Madrid,\acr Spain.}
\email{dalfaya@comillas.edu, jrodrigo@comillas.edu}

\address{\llap{**\,}Department of Quantitative Methods,\acr
ICADE, Comillas Pontifical University,\acr
C/Alberto Aguilera 23, 28015 Madrid,\acr Spain.}
\email{lacalvo@comillas.edu, asrinivasan@icade.comillas.edu}

\address{\llap{***\,}Institute for Research in Technology,\acr
ICAI School of Engineering,\acr
Comillas Pontifical University,\acr
C/Santa Cruz de Marcenado 26, 28015 Madrid,\acr Spain.}
\email{dalfaya@comillas.edu, alex.m.guinea@alu.comillas.edu}

\thanks{\textit{Acknowledgements}. This research was supported by project CIAMOD (Applications of computational methods and artificial intelligence to the study of moduli spaces, project PP2023\_9) funded by Convocatoria de Financiaci\'on de Proyectos de Investigaci\'on Propios 2023, Universidad Pontificia Comillas, and by grant PID2022-142024NB-I00 funded by MCIN/AEI/10.13039/501100011033.}


\begin{abstract}
We classify all solution triples with $k$-Fibonacci components to the equation $x^2+y^2+z^2=3xyz+m,$ where $m$ is a positive integer and $k\geq 2$. As a result,  for $m=8$, we have the Markoff triples with Pell components $(F_2(2), F_2(2n), F_2(2n+2))$, for $n\geq 1$. For all other $m$ there exists at most one such ordered triple, except when  $k=3,$ $a$ is odd, $b$ is even and $b\geq a+3$,  where $(F_3(a),F_3(b),F_3(a+b))$ and $(F_3(a+1),F_3(b-1),F_3(a+b))$ share the same $m$.
\end{abstract}

\maketitle

\providecommand{\keywords}[1]
{
  \small	
  \textbf{\textit{Keywords--}} #1
}

\setcounter{section}{0}
\section{Introduction}

In the realm of number theory, Markoff $m$-triples represent an interesting area of exploration. These triples are positive integer solutions to the Markoff $m$-equation
\begin{equation}\label{Markoffeq}
x^2+y^2+z^2=3xyz+m,
\end{equation}
where $m$ is a positive integer.
The case $m=0$ corresponds to the original equation studied by A. A. Markoff in \cite{M1,M2}, where it was proved that all the solution triples are distributed in a unique tree. Some of its branches are interesting families of numbers: Fibonacci, Pell, etc. Many authors studied generalizations of this equation (\cite{Mor}, \cite{GS}, \cite{SC}) and noticed that, depending on $m,$ there could exist one, multiple trees or none at all. In particular, in \cite{SC} it is proved that the number of trees, for every $m>0,$ is equal to the number of Markoff $m$-triples $(x,y,z)$ that are minimal, that is to say, those that satisfy the inequality
\begin{equation}\label{Minimals}
z\geq 3xy.
\end{equation}

In this paper, we study Markoff $m$-triples with $k$-Fibonacci components, i.e. solutions of the Markoff $m$-equation (\ref{Markoffeq}), such that all its components are $k$-Fibonacci numbers. These numbers are defined recursively for every positive integer $k$ as follows
\begin{equation}
\label{df:k_fibonacci}
\left\{
\begin{array}{l}
F_k(0)=0\\
F_k(1)=1\\
F_k(n)=k F_k(n-1)+F_k(n-2), \quad \forall n\geq 2.\\
\end{array}\right .
\end{equation}
When $k=1,$  the sequence corresponds to the classic Fibonacci numbers, and for $k=2$, it yields Pell numbers. Some particular cases of Markoff $m$-triples with $k$-Fibonacci components  have already been studied: $(k=1,\, m=0),$ was studied in \cite{LS};  $(k=2,\, m=0),$ was examined in \cite{KST}; $(k>1,\, m=0),$ was treated in \cite{Gom}; the case $m=0$, with Lucas sequences in \cite{AL},\cite{RSP} and, finally, the case $(k=1,\, m>0)$ was dealt with in \cite{ACMRS}. Because of this, henceforth, we will assume that $m>0$ and $k\geq 2.$

In this work, we classify all Markoff $m$-triples with $k$-Fibonacci components, dividing our analysis first into non-minimal triples and then into minimal ones. Specifically, our main results are the following.

\begin{theorem} \label{th1}
Every non-minimal Markoff $m$-triple with $k$-Fibonacci components and $m>0$ is a Markoff $8$-triple of the form $(F_2(2), F_2(2n),F_2(2n+2)),$ for $n\geq 2$.
\end{theorem}

In particular, the non-minimal Markoff $m$-triples with $k$-Fibonacci components are situated on the upper branch of the 8-tree with minimal triple $(2,2,12).$ The triples in this branch are composed of Pell numbers, as shown in Figure \ref{8-markovtree}. 

\begin{figure}[H]
\centering
\begin{tikzpicture}[grow'=right, edge from parent/.style={draw, -latex}]
\tikzset{
  level 1/.style={level distance=2.3cm, sibling distance=0.5cm},
  level 2/.style={level distance=2.7cm, sibling distance=0.5cm},
  level 3/.style={level distance=3.2cm, sibling distance=0.5cm},
  level 4/.style={level distance=4cm, sibling distance=0.3cm},
  level 5/.style={level distance=4.5cm, sibling distance=0.25cm}
}
\Tree 
[. (2,2,12)
    [.\textbf{(2,12,70)}
        [.\textbf{(2,70,408)}
            [.\textbf{(2,408,2378)}
                [.\textbf{(2,2378,13860)} ]
                [.(408,2378,2910670) ]
            ]
            [.(70,408,85678)
                [.(70,85678,17991972) ]
                [. (408,85678,104869802) ]
            ]
        ]
        [.(12,70,2518)
            [.(12,2518,90578)
                [.(12,90578,3258290) ]
                [.(2518,90578,684226200) ]
            ]
            [.(70,2518,528768)
                [.(70,528768,111038762) ]
                [.(2518,528768,3994313402) ]
            ]
        ]
    ]
]
\end{tikzpicture}
\caption{Beginning of the Markoff $8$-tree with minimal triple $(2,2,12).$ The sequence of non-minimal $8$-Markoff triples with $2$-Fibonacci components (Pell components) is represented in bold.}
\label{8-markovtree}
\end{figure}
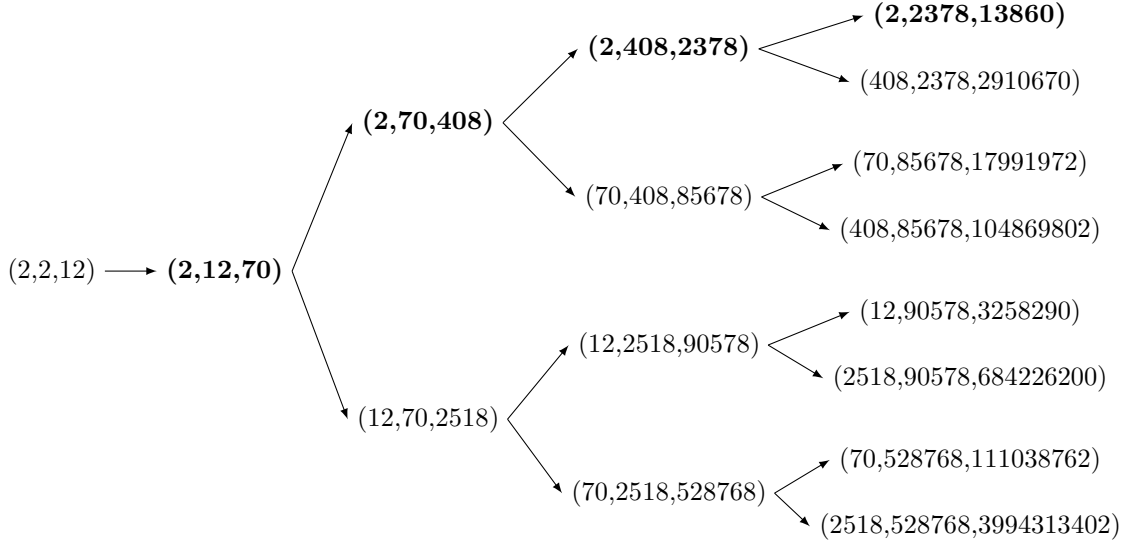

\begin{theorem} \label{th2}
If $m>0$ admits a minimal Markoff $m$-triple with $k$-Fibonacci components, then it is unique, except 
for $k=3$ and all  pairs of triples $(F_3(a),F_3(b),F_3(a+b)), \,\,(F_3(a+1),F_3(b-1),F_3(a+b)),$ for $a$ odd and $b$ even with $b\ge a+3$.
\end{theorem}

The paper is structured as follows. Section \ref{section:preliminary} provides certain identities and inequalities satisfied by $k$-Fibonacci numbers which will be useful in the next sections. Although most of them are well known \cite{F}, \cite{V}, \cite{Ko}, we have included proofs for some of them for the sake of completeness. In Section \ref{section:non-minimal}, we prove Theorem \ref{th1} and in Section \ref{section:minimal}, Theorem \ref{th2}. The strategy to obtain uniqueness in minimal Markoff $m$-triples $(F_k(a), F_k(b), F_k(c))$ except in the case $k=3, c=a+b, a$ odd, $b$ even and $b\geq a+3$ involves proving that any pair of such triples which share the same $m$ must have the same third component $c$, and the sum $a + b$ should be constant (see Lemmas \ref{lemma:c=c'} and \ref{lemma:a+b=a'+b'}). These two lemmas, in turn, follow from Lemma \ref{lemma:KaramataBound}, which computes a lower bound for the $m$ associated with an $m$-triple $(F_k(a), F_k(b), F_k(c))$ in terms of $k$ and $c$.

%
%
%

\section{Some preliminary results on \texorpdfstring{$k$}{k}-Fibonacci numbers}
\label{section:preliminary}
For any $k>0$ and $n\geq 0$ the $n$-th term of the sequence of $k$-Fibonacci numbers, defined in equation \eqref{df:k_fibonacci}, can be obtained using Binet's formula
\begin{equation}
\label{eq:Binet}
F_k(n)=\frac{\alpha_k^n-\bar{\alpha}_k^n}{D_k},
\end{equation}
where $\alpha_k$ and $\bar{\alpha}_k$ are the roots of the characteristic polynomial of the recurrence $\alpha^2-k\alpha-1=0$ and $D_k=\alpha_k-\bar{\alpha}_k$.  Concretely,
$$\alpha_k = \frac{k+\sqrt{k^2+4}}{2}, \quad \bar{\alpha}_k=\frac{k-\sqrt{k^2+4}}{2}, \quad D_k=\alpha_k-\bar{\alpha}_k = \sqrt{k^2+4}.$$
The above formula is well known; for a proof the reader may consult Theorem 7.4 of \cite{Ko}. It is a consequence of the fact that any $k$-Fibonacci number is defined by recurrence relation \eqref{df:k_fibonacci} and it is a solution of the corresponding second-order finite difference equation.
Notice that $\alpha_k\bar{\alpha}_k=-1$. In particular, for $k=1$, $\alpha_1=\varphi$ and $D_1=\sqrt{5}$, we have the classical Binet formula for the Fibonacci numbers, where $\varphi$ represents the Golden Ratio.

\begin{lemma}[Generalization of Vajda's Identity for $k$-Fibonacci numbers] \label{lemma:Vajda} For any positive numbers $i,j, k,$ $$F_k(n+i)F_k(n+j)-F_k(n)F_k(n+i+j)=(-1)^nF_k(i)F_k(j).$$
\end{lemma}

\begin{proof}
Multiplying the left hand side by $D_k^2$ and using Binet's formula (\ref{eq:Binet}) and the fact that $\alpha_k\bar{\alpha}_k=-1$ yields
\begin{multline*}
    D_k^2\left(F_k(n+i)F_k(n+j)-F_k(n)F_k(n+i+j)\right) = (\alpha_k^{n+i}-\bar{\alpha}_k^{n+i})(\alpha_k^{n+j}-\bar{\alpha}_j^{n+j})-(\alpha_k^n-\bar{\alpha}_k^n)(\alpha_k^{n+i+j}-\bar{\alpha}_k^{n+i+j})\\
    =\alpha_k^{2n+i+j}-(-1)^n \alpha_k^i \bar{\alpha}_k^j - (-1)^n \bar{\alpha}_k^i \alpha_k^j + \bar{\alpha}_k^{2n+i+j} - \alpha_k^{2n+i+j}+ (-1)^n \alpha_k^{i+j} + (-1)^n \bar{\alpha}_k^{i+j}-\bar{\alpha}_k^{2n+i+j}\\
    =(-1)^n(\alpha_k^i-\bar{\alpha}_k^i)(\alpha_k^j - \bar{\alpha}_k^j)= D_k^2\left ((-1)^nF_k(i)F_k(j) \right).
\end{multline*}
\end{proof}

\begin{corol}\label{cor:identities}
The following identities hold for any integers $a,b,n\geq 1:$
\begin{gather}
    F_k(a+b)=F_k(a+1)F_k(b)+F_k(a)F_k(b-1) \label{eq:sum}\, \\
    F_k(a)\leq \frac{1}{k} F_k(a+1)  \label{eq:basic_bound}\, \\
    F_k(a) F_k(b)\leq F_k(a+b-1)  \label{eq:bound_product}\, \\
    F_k(a+b-1)\le F_k(a)F_k(b)\left(1+\frac{1}{k^2}\right)   \label{eq:bound_product2}\, \\
    (\text{D'Ocagne\,\, identity}) \,\,(-1)^aF_k(b-a) = F_k(b)F_k(a+1)-F_k(b+1)F_k(a) \label{eq:DOcagne}\,\\
    (\text{Catalan\,\, identity}) \,\,F_k(n)^2=F_k(n+r)F_k(n-r)+(-1)^{n-r}F_k(r)^2 \label{eq:Catalan}\,\\
    (\text{Simson\,\, identity}) \,\,F_k(n)^2=F_k(n+1)F_k(n-1)-(-1)^{n} \label{eq:Simson}\,.
\end{gather}
Moreover, equality holds in the following cases:
\begin{enumerate}
    \item The equality in \eqref{eq:basic_bound} is only attained if $a=1$.
    \item The equality in \eqref{eq:bound_product} is only attained if $a=1$ or $b=1$.
    \item The equality in \eqref{eq:bound_product2} is only attained if $a=b=2$.
\end{enumerate}
\end{corol}

\begin{proof}
For \eqref{eq:sum}, take $n=1$, $i=a$ and $j+1=b$ in the previous lemma.

For \eqref{eq:basic_bound}, we have
$$F_k(a+1)=kF_k(a)+F_k(a-1)\ge kF_k(a)$$
and equality is only attained if $F_k(a-1)=0$, i.e., if $a=1$.

For \eqref{eq:bound_product}, substitute $a$ by $a-1$ in identity \eqref{eq:sum}. Then
$$F_k(a+b-1)=F_k(a)F_k(b)+F_k(a-1)F_k(b-1)\ge F_k(a)F_k(b).$$
Equality is only attained if $F_k(a-1)=0$ or $F_k(b-1)=0$, i.e., if $a=1$ or $b=1$.

For \eqref{eq:bound_product2}, substitute $a$ by $a-1$ in identity \eqref{eq:sum}. Then
$$F_k(a+b-1)=F_k(a)F_k(b)+F_k(a-1)F_k(b-1)\le F_k(a)F_k(b)\left(1+\frac{1}{k^2}\right).$$
Equality is only attained if $F_k(a-1)=\frac{1}{k}F_k(a)$ and $F_k(b-1)=\frac{1}{k}F_k(b)$, which only happens if $a=b=2$.

For the D'Ocagne identity \eqref{eq:DOcagne}, take $n=a$, $i=b-a$, $j=1$ in the previous lemma.

For Catalan's identity \eqref{eq:Catalan}, take $n=n-r$, $i=j=r$ in the previous lemma.

Finally, for the Simson identity \eqref{eq:Simson}, take $r=1$ in the Catalan identity \eqref{eq:Catalan}.

\end{proof}

\begin{lemma}
\label{lemma:sumSquares}
For integers $k\ge 1$ and  $N\ge 0,$
$$\sum_{n=0}^N F_k(n)^2 = \frac{1}{k} F_k(N)F_k(N+1)\,.$$
\end{lemma}

\begin{proof}
We will use induction to prove the result. For $N=0,$ the identity is true because $F_k(0)=0$. Assuming that the result holds for some $N$, we will prove it for $N+1$. We begin with the following equation
$$\frac{1}{k}F_k(N+1)F_k(N+2)=\frac{1}{k} F_k(N+1)(kF_k(N+1)+F_k(N))=F_k(N+1)^2+\frac{1}{k}F_k(N)F_k(N+1)\, .$$
And, by the induction hypothesis, we have
$$F_k(N+1)^2+\frac{1}{k}F_k(N)F_k(N+1)=F_k(N+1)^2+\sum_{n=0}^N F_k(n)^2 =\sum_{n=0}^{N+1} F_k(n)^2\,,$$
which completes the proof.
\end{proof}

\begin{lemma}
\label{lemma:F(2b-2)lesssquares}
If $k\ge 4$ and $n\geq 1,$ then $4F_k(2n-2) \le F_k(n)^2.$
\end{lemma}

\begin{proof} For $n=1,$ the inequality becomes $0=4F_k(0)\le F_k(1)=1$, hence the result holds. Assume that $n\ge 2$. Taking $a=b=n-1$ in equation \eqref{eq:sum}, and then multiplying by four, we obtain 
\begin{equation}\label{4f_k}
4F_k(2n-2)=4F_k(n-1) ( F_k(n)+F_k(n-2)).
\end{equation}

\noindent If $k\ge 5$, then $4 F_k(n-1)\le 4/5 F_k(n)$ and
$F_k(n-2)<1/4 F_k(n).$ Combining both inequalities, we get 
$$4  F_k(n-1) ( F_k(n)+ F_k(n-2))<F_k(n)^2.$$
The above inequality and \eqref{4f_k} prove the lemma for $k\geq 5.$ In the case $k=4,$ using again \eqref{4f_k}, we have 
\begin{multline*}4 F_4(2n-2)= 4 F_4(n-1) ( F_4(n)+ F_4(n-2))=(F_4(n)-F_4(n-2)) \left( F_4(n)+ F_4(n-2)\right)=\\
F_4(n)^2-F_4(n-2)^2\le F_4(n)^2,\end{multline*}
which proves the result.
\end{proof}

\begin{lemma}
\label{lemma:F(n+m)<=3F(n)F(m)}
Let $a,b,c\ge 1$. Then
\begin{align}
F_2(c)\ge  3F_2(a)F_2(b)& \quad \text{if and only if} \quad c\ge a+b+1 \text{ or } (a,b,c)=(2,2,4), \text{ and } \label{case_2}\\
F_k(c)\ge 3F_k(a)F_k(b)& \quad \text{if and only if} \quad c\ge a+b,\,\, \text{ for all }k\geq 3\label{eq:F(n+m)<=3F(n)F(m)3}\, .
\end{align}
Equality is only attained if $k=2$ and $(a,b,c)=(2,2,4)$, or if $k=3$ and $(a,b,c)=(1,1,2)$.
\end{lemma}

\begin{proof}
We first prove \eqref{case_2}. By identity \eqref{eq:sum}, we have that
\begin{multline}\label{cota_max}F_2(a+b+1)=F_2(a+1)F_2(b+1)+F_2(a)F_2(b) = (2F_2(a)+F_2(a-1))(2F_2(b)+\\F_2(b-1))+F_2(a)F_2(b)  \ge (2^2+1)F_2(a)F_2(b)>3F_2(a)F_2(b)\,.
\end{multline}

\noindent On the other hand, 
\begin{equation*}\frac{F_2(a+b)}{F_2(a)F_2(b)}=\frac{F_2(a+1)F_2(b)+F_2(a)F_2(b-1)}{F_2(a)F_2(b)}=\frac{F_2(a+1)}{F_2(a)}+\frac{F_2(b-1)}{F_2(b)}.\end{equation*}
It is known that successive quotients of Pell numbers $F_2(n+1)/F_2(n)$ form an oscillating sequence converging to $\alpha_2$, where the sequence of even terms is decreasing and the sequence of odd terms is increasing. As a consequence, the maximum of $F_2(a+1)/F_2(a)$ is $\frac{5}{2}$ and it is attained only at $a=2$, and the maximum of $F_2(b-1)/F_2(b)$ is $\frac{1}{2}$ and it is attained only at $b=2$. Thus,
\begin{equation} \label{cota_inf}\frac{F_2(a+b)}{F_2(a)F_2(b)}=\frac{F_2(a+1)}{F_2(a)}+\frac{F_2(b-1)}{F_2(b)}\le \frac{5}{2}+\frac{1}{2}=3\end{equation}
and  equality is only attained at $(a,b)=(2,2)$.
Combining \eqref{cota_max} and \eqref{cota_inf} and using the fact that the function $F_2(c)$ is strictly increasing in $c$, we see that  \eqref{case_2} holds. 

Finally, we prove \eqref{eq:F(n+m)<=3F(n)F(m)3}. By using again \eqref{eq:sum}, if $k\ge 3$
$$F_k(a+b)=F_k(a+1)F_k(b)+F_k(a)F_k(b-1)=kF_k(a)F_k(b)+F_k(a-1)F_k(b)+F_k(a)F_k(b-1)\ge 3F_k(a)F_k(b),\,$$
with equality if and only if $k=3,$ $F_k(a-1)=0$ and $F_k(b-1)=0$, i.e., if $a=b=1$.
Additionally, for all $k\ge 3$ it follows that
$$F_k(a+b-1)=F_k(a)F_k(b)+F_k(a-1)F_k(b-1)\le 2F_k(a)F_k(b)<3F_k(a)F_k(b).$$
By the two previous inequalities and since the function $F_k(c)$ is strictly increasing in $c$, it follows that \eqref{eq:F(n+m)<=3F(n)F(m)3} holds.

\end{proof}

\section{Non-minimal case}\label{section:non-minimal}

Recall that a Markoff $m$-triple $(x,y,z)$ is a positive integer solution triple of the Markoff $m$-equation \eqref{Markoffeq}, where $m$ is a positive integer. Henceforth, we assume that the triple is ordered, i.e. $x\leq y\leq z$. For positive integers $a, b, c$, we shall denote
$$m_k(a,b,c)=F_k(a)^2+F_k(b)^2+F_k(c)^2-3F_k(a)F_k(b)F_k(c),$$
so that $(F_k(a),F_k(b),F_k(c))$ is a Markoff $m$-triple with $k$-Fibonacci components if and only if $m_k(a,b,c)>0$. In this section, after deriving conditions on $(a,b,c)$ for which $m_k(a,b,c)\leq 0$, as a straightforward consequence, we prove Theorem \ref{th1}, showing that there exists only one branch of non-minimal Markoff $m$-triples with $k$-Fibonacci components.  Note that we consider $k\geq2,$ since the case $k=1$ was previously treated in \cite{ACMRS}.

\begin{lemma}\label{lemma:m_negativo}
\leavevmode
\begin{enumerate} 
    \item For $a \geq 3$, if $c \leq a + b$, then $m_2(a,b,c) \leq 0$.
    \item For $a \geq 1$, if $c < a + b$, then $m_k(a,b,c) \leq 0,$ for all $k \geq 3$.
\end{enumerate}
\end{lemma}

\begin{proof}
\noindent We start with (2). We have
\begin{equation} \label{eq:first}
2F_k(a+1) = 2(kF_k(a) + F_k(a-1)) \leq 2(k+1)F_k(a) \leq 3kF_k(a),
\end{equation}
for $k \geq 2$. Next, from equation \eqref{eq:sum} and \eqref{eq:first} above, we obtain
\begin{equation} \label{eq:second}
F_k(a+b) \leq 2F_k(a+1)F_k(b) \leq 3kF_k(a)F_k(b).
\end{equation}
Also, since $c \leq a+b-1,$ from \eqref{eq:second} above,
\begin{equation} \label{eq:third}
F_k(c+1)F_k(c) \leq F_k(a+b)F_k(c) \leq 3kF_k(a)F_k(b)F_k(c).
\end{equation}

\noindent Now, by Lemma \ref{lemma:sumSquares}, assuming $a, b, c$ distinct or $a=b<c-1,$ we have
\begin{equation} \label{eq:lemma}
F_k(a)^2 + F_k(b)^2 + F_k(c)^2 \leq \frac{F_k(c+1)F_k(c)}{k}.
\end{equation}

\noindent Then,  \eqref{eq:third} and \eqref{eq:lemma} yield
\begin{equation*}
F_k(a)^2 + F_k(b)^2 + F_k(c)^2 \leq 3F_k(a)F_k(b)F_k(c), \quad 
\end{equation*}
which is equivalent to $m_k(a,b,c)\leq 0.$

Observe that in the case $a\le b=c,$ we trivially have $m_k(a,b,c) \le 0.$
Next, we prove the remaining case $a=b=c-1.$ As $F_k(c)\leq (k+1)F_k(c-1)$, we have
\begin{equation}\label{eq:four}2F_k(c-1)^2+F_k(c)^2\leq 2F_k(c-1)^2+(k+1)^2F_k(c-1)^2=F_k(c-1)^2 \left(2+(k+1)^2\right).\end{equation}
Since $c\leq a+b-1= 2(c-1)-1,$  we can suppose that $c\geq 3,$ which leads to
$$2+(k+1)^2< 3(k^2+1)\leq 3F_k(c).$$ As a result,
\begin{equation}\label{eq:cinco}F_k(c-1)^2 \left(2+(k+1)^2\right)< F_k(c-1)^2 \,3\,F_k(c). \end{equation}
Combining equations \eqref{eq:four} and \eqref{eq:cinco}, we obtain
$$2F_k(c-1)^2+F_k(c)^2 < 3\,F_k(c-1)^2\, F_k(c),$$
which can also be expressed as $m_k(c-1,c-1,c)<0.$


Finally, we prove (1). The only case to be checked is \( c = a + b \) because the proof above is valid if $c\geq a+b+1$. We aim to prove
\begin{equation*}
F_2(a)^2 + F_2(b)^2 + F_2(a+b)^2 \leq 3F_2(a)F_2(b)F_2(a+b).
\end{equation*}
Adding $2F_2(a)F_2(b)$ on both sides,
\begin{equation*}
\left( F_2(a)+ F_2(b)\right)^2 + F_2(a+b)^2 \leq F_2(a)F_2(b)\left(3F_2(a+b)+2\right).
\end{equation*}
Since $\left( F_2(a)+ F_2(b)\right)^2\leq 4F_2(b)^2$, it suffices to prove
\begin{equation*}
4F_2(b)^2+F_2(a+b)^2\leq 3F_2(a) F_2(b) F_2(a+b).
\end{equation*}
Rearranging terms,
\begin{equation*}
4F_2(b)^2\leq  F_2(a+b)\left(3F_2(a) F_2(b)-F_2(a+b)\right).
\end{equation*}
Developing $F_2(a+b)$ on the right-hand side, using \eqref{eq:sum},
\begin{equation*}
4F_2(b)^2\leq  F_2(a+b)\left( 3F_2(a) F_2(b)-F_2(a+1)F_2(b)-F_2(a)F_2(b-1) \right).
\end{equation*}
Using  $3F_2(a)-F_2(a+1)=F_2(a-1)+F_2(a-2)$, we obtain
\begin{equation*}
4F_2(b)^2\leq  F_2(a+b)\left(F_2(b)(F_2(a-1)+F_2(a-2))-F_2(a)F_2(b-1)\right),
\end{equation*}
and thus, reordering terms on the right-hand side we have
\begin{equation*}
4F_2(b)^2\leq  F_2(a+b)\left(F_2(b)F_2(a-2)+F_2(b) F_2(a-1)-F_2(a)F_2(b-1)\right).
\end{equation*}
Now, applying D'Ocagne identity \eqref{eq:DOcagne} to $a-1$ and $b-1$,
\begin{equation}\label{probar_casos}
4F_2(b)^2\leq  F_2(a+b)\left(F_2(b) F_2(a-2)+(-1)^aF_2(b-a)\right).
\end{equation}
To prove the inequality above, we distinguish two cases: $a$ being even and odd.  If $a$ is even, since $a\geq 4,$ then $F_2(a-2)\geq 2$ and $F_2(a+b)\geq 4F_2(b)$. Consequently, $$4F_2(b)\leq F_2(a+b)F_2(a-2)$$
and \eqref{probar_casos} holds. If $a$ is odd, since $a\geq3,$ we have $12F_2(b)\leq F_2(a+b)$, and for proving \eqref{probar_casos} it is enough to prove
$$F_2(b)\leq 3F_2(b)F_2(a-2)-3F_2(b-a).$$
in other words,
$$F_2(b)+3F_2(b-a)\leq 3F_2(b)F_2(a-2)$$
and this holds because $3F_2(b-a)\leq 3F_2(b-3)\leq \frac{F_2(b)}{4}$ and $F_2(a-2)\geq 1$.

\end{proof}

\begin{lemma} \label{lemma_case1} The following hold.
\leavevmode
\begin{enumerate} 
    \item $m_2(1,b,b+1)\leq 0,$ for any $b,$ and  equality holds only for $b=1,2.$
    \item $m_2(2,b,b+1)< 0,$ for any $b\geq 2.$
\end{enumerate}
\end{lemma}
\begin{proof}

For (1), it suffices to prove
$$
1+F_2(b)^2+F_2(b+1)^2\leq 3F_2(b)F_2(b+1).
$$
If $b=1,$  the equation above holds as an equality. If $b>1$, by applying  Lemma \ref{lemma:sumSquares} to the left-hand side,  the above is equivalent to
\begin{equation}\label{replicar} \frac{1}{2}F_2(b+1)F_2(b+2)\leq 3F_2(b)F_2(b+1).\end{equation}
Equivalently,
$$F_2(b+1)(2F_2(b+1)+F_2(b))\leq 6F_2(b)F_2(b+1).$$
Dividing by $F_2(b+1)\neq 0,$ we obtain $2F_2(b+1)\leq 5F_2(b),$ but this inequality holds because $2F_2(b+1)=4F_2(b)+2F_2(b-1)$ and $F_2(b)\geq 2F_2(b-1).$ In this case, equality is only achieved when $b=2.$

Next, (2) is equivalent to  
$$
4+F_2(b)^2+F_2(b+1)^2< 6F_2(b)F_2(b+1).
$$
If $b=2$, we can verify the above inequality numerically $(4+4+25< 60)$. For $b>2$, by  Lemma \ref{lemma:sumSquares}, and equation
\eqref{replicar}, we see that the above holds.
\end{proof}

\begin{theorem}[Theorem \ref{th1} of the Introduction] 
Every non-minimal Markoff $m$-triple with $k$-Fibonacci components is an Markoff $8$-triple of the form $(F_2(2), F_2(2n),F_2(2n+2)),$ for $n\geq 2$.
\end{theorem}

\begin{proof} 
First, we start with the case $k\geq 3.$ If a Markoff $m$-triple with $k$-Fibonacci components $(F_k(a),F_k(b),F_k(c))$ is not minimal then $c< a+b$, by Lemma \ref{lemma:F(n+m)<=3F(n)F(m)}. However, by Lemma \ref{lemma:m_negativo} (2), for $k\ge 3$ this restriction implies that $m_k(a,b,c)\le 0.$ Therefore, non-minimal Markoff $m$-triples with $k$-Fibonacci components do not exist for $k\geq3.$

In the case $k=2,$ if a Markoff $m$-triple with $2$-Fibonacci components $(F_2(a),F_2(b),F_2(c))$ is not minimal, then $c\leq a+b,$ by Lemma \ref{lemma:F(n+m)<=3F(n)F(m)}.   This restriction forces $F_2(a)$ to be equal to $1$ or $2$, because of Lemma \ref{lemma:m_negativo} (1). If $F_2(a)=1,$ then $a=1$ and $c\leq b+1$. In the case $b=c$ it is obvious than $m_2(1,b,b)\leq 0$ and in the case $c=b+1,$ it follows that $m_2(1,b,b+1)\leq 0$ by Lemma \ref{lemma_case1} (1). Finally, if $F_2(a)=2=F_2(2),$ then $a=2,$ and $c\leq 2+b.$ Hence by  Lemma \ref{lemma_case1} (2), the triple is of the form $(2,b,b+2)$. Now, we prove that $b$ is an even number. Indeed,
\begin{multline}
\label{eq:2-non-minimal}
m_2(2,b,b+2)=4+F_2(b)^2+F_2(b+2)^2-6F_2(b)F_2(b+2)
=4+(F_2(b+2)-F_2(b))^2-4F_2(b)F_2(b+2)\\=4+4F_2(b+1)^2-4F_2(b)F_2(b+2)=4(1-(-1)^{b+1})
\end{multline}
is positive if and only if $b$ is even, where the last equality is a consequence of the Simson identity (\ref{eq:Simson}).
As a result, all the triples of the form $(F_2(2), F_2(2n),F_2(2n+2)),$ for $n\geq 1$ are $8$-triples and it is straightforward to check that they all lie in a branch of the Markoff $8$-tree with minimal triple $(2,2,12)$ (See Fig. \ref{8-markovtree}). For $m=8,$ this tree is unique because there are no more minimal triples than $(2,2,12)$  as shown in Table 1 of \cite{SC}.

\end{proof}

%
%

\section{ Minimal case}\label{section:minimal}

We recall that if $(x,y,z)$ is a minimal Markoff $m$-triple, i.e. a solution of the Markoff $m$-equation \eqref{Markoffeq}, with $z\geq 3xy,$ then 
$$m=z(z-3xy)+x^2+y^2>0.$$ Let $a,b$ be any pair of positive integers with $a\leq b$ and let $c=a+b+t.$ By Lemma \ref{lemma:F(n+m)<=3F(n)F(m)}, if \( t \geq 1 \) for \( k = 2 \),  or \( t \geq 0 \) for \( k \geq 3 \), then \( (F_k(a), F_k(b), F_k(c)) \) is minimal, therefore $m_k(a,b,c)>0$. Consequently, there exists an infinite number of minimal Markoff triples with $k$-Fibonacci components. Clearly they cannot all correspond to a finite number of values of $m$, as the number of minimal triples is finite for each $m$ \cite{SC}. Hence there are infinitely many values of $m$ that admit minimal Markoff $m$-triples with $k$-Fibonacci components.  In the rest of the section, we will prove that any $m>0$ admits at most one minimal Markoff $m$-triple with $k$-Fibonacci components,  except when $k=3,$ $c=a+b,$ $a$ is odd, $b$ is even and $b\geq a+3$, where $m_3(a,b,a+b)$ admits two such triples.

\begin{lemma}
\label{lemma:KaramataBound}
Let $1\le a\le b$.  Suppose that $k=2$ and $c=a+b+1$, or $k\ge 3$ and $c=a+b$. Then
$$m_k(a,b,c)> L_k \frac{\alpha_k^{2c}}{D_k^2},$$
where $D_k=\alpha_k-\bar{\alpha_k} =\sqrt{k^2+4}$ and 
\begin{align*}
L_2=&\left(1-\frac{3}{D_2}\alpha_2^{-1}\right) + 2\left(1-\frac{3}{D_2}\alpha_2\right)\alpha_2^{-4}-\left(6+\frac{3}{D_2}\alpha_2+\frac{9}{D_2}\right)\alpha_2^{-6},\\
L_3=&\left(1-\frac{3}{D_3}\right)(1+2\alpha_3^{-2})-\left(6+\frac{12}{D_k}\right)\alpha_2^{-4},\\
L_k=& 1-\frac{3}{D_k}, \qquad \forall k\ge 4.
\end{align*}
\end{lemma}

\begin{proof}
Using Binet's formula (\ref{eq:Binet}) and taking into account that $\alpha_k\bar{\alpha}_k=-1$, it follows that for any $k\ge 1$
$$F_k(n)^2=\frac{1}{D_k^2}\left(\alpha_k^{2n}+\alpha_k^{-2n}-2\cdot(-1)^n\right)>\frac{1}{D_k^2}\left(\alpha_k^{2n}-2\right).$$
If $k=2$ and $b=c-1-a$, we have
\begin{multline*}
m_2(a,b,c)=F_2(c)^2+F_2(c-1-a)^2+F_2(a)^2-3F_2(c)F_2(c-1-a)F_2(a)\\
>\frac{1}{D_2^2}\left(\alpha_2^{2c}+\alpha_2^{2c-2-2a}+\alpha_2^{2a}-6\right)-\frac{3}{D_2^3} (\alpha_2^{c}-\bar{\alpha}_2^{c})(\alpha_2^{c-1-a}-\bar{\alpha}_2^{c-1-a})(\alpha_2^{a}-\bar{\alpha}_2^{a})\, .
\end{multline*}
As $c=a+b+1>1$ and $\alpha_2\bar{\alpha}_2=-1$, we conclude that
\begin{multline*}
    (\alpha_2^{c}-\bar{\alpha}_2^{c})(\alpha_2^{c-1-a}-\bar{\alpha}_2^{c-1-a})(\alpha_2^{a}-\bar{\alpha}_2^{a})\le
    (\alpha_2^{c}+\alpha_2^{-c})(\alpha_2^{c-1-a}+\alpha_2^{a-c+1})(\alpha_2^{a}+\alpha_2^{-a})=\\
    \alpha_2^{2c-1}+\alpha_2^{2c-1-2a}+\alpha_2^{2a+1}+\alpha_2+\alpha_2^{-1}+\alpha_2^{-2a-1}+\alpha_2^{2a-2c+1}+\alpha_2^{-2c+1}
    <\alpha_2^{2c-1}+\alpha_2^{2c-1-2a}+\alpha_2^{2a+1} +\alpha_2+3.
\end{multline*}
Hence
\begin{multline*}
m_2(a,b,c)> \frac{1}{D_2^2}\left(\alpha_2^{2c}+\alpha_2^{2c-2-2a}+\alpha_2^{2a}-6\right)-\frac{3}{D_2^3} (\alpha_2^{2c-1}+\alpha_2^{2c-1-2a}+\alpha_2^{2a+1}+\alpha_2+3)
\\=\frac{1}{D_2^2}\alpha_2^{2c}\left[\left(1-\frac{3}{D_2}\alpha_2^{-1}\right)+\left(1-\frac{3}{D_2}\alpha_2\right)\left(\alpha_2^{-2-2a}+\alpha_2^{2a-2c}\right)-\left(6+\frac{3}{D_2}\alpha_2+\frac{9}{D_2}\right)\alpha_2^{-2c}\right]\, .
\end{multline*}
As $f(x)=\alpha_2^x$ is a convex function, $c>1$ and $a\ge 1$, by applying Karamata's inequality \cite{K}, we obtain
\begin{equation}
\label{eq:karamata}
\alpha_2^{-2-2a}+\alpha_2^{2a-2c}\le \alpha_2^{-2-2}+\alpha_2^{2-2c}=\alpha_2^{-4}+\alpha_2^{2-2c}\, .
\end{equation}
Since 
$$1-\frac{3}{D_2}\alpha_2=1-\frac{6+3\sqrt{8}}{2\sqrt{8}}<1-\frac{3}{2}<0$$
 and $c\ge a+b+1\ge 3$, we have
\begin{multline*}
    m_2(a,b,c)>\frac{1}{D_2^2}\alpha_2^{2c}\left[\left(1-\frac{3}{D_2}\alpha_2^{-1}\right)+\left(1-\frac{3}{D_2}\alpha_2\right)\left(\alpha_2^{-2-2a}+\alpha_2^{2a-2c}\right)-\left(6+\frac{3}{D_2}\alpha_2+\frac{9}{D_2}\right)\alpha_2^{-2c}\right]\\
    \ge \frac{1}{D_2^2}\alpha_2^{2c}\left[\left(1-\frac{3}{D_2}\alpha_2^{-1}\right) + \left(1-\frac{3}{D_2}\alpha_2\right)(\alpha_2^{-4}+\alpha_2^{2-2c})-\left(6+\frac{3}{D_2}\alpha_2+\frac{9}{D_2}\right)\alpha_2^{-2c}\right] \ge L_2\frac{1}{D_2^2}\alpha_2^{2c},
\end{multline*}
as the coefficient of $\alpha_2^{-2c}$ is clearly negative in the previous expression, and therefore its minimum for $c\ge 3$ is attained at $c=3$.

Analogously, if we assume that $k\ge 3$ and  $c=a+b$, we have
\begin{multline*}
    (\alpha_k^{c}-\bar{\alpha}_k^{c})(\alpha_k^{c-a}-\bar{\alpha}_k^{c-a})(\alpha_k^{a}-\bar{\alpha}_k^{a})\le
    (\alpha_k^{c}+\alpha_k^{-c})(\alpha_k^{c-a}+\alpha_k^{a-c})(\alpha_k^{a}+\alpha_k^{-a})=\\
    \alpha_k^{2c}+\alpha_k^{2c-2a}+\alpha_k^{2a}+2+\alpha_k^{-2a}+\alpha_k^{2a-2c}+\alpha_k^{-2c}
    <\alpha_k^{2c}+\alpha_k^{2c-2a}+\alpha_k^{2a} +4.
\end{multline*}
Hence
\begin{multline*}
m_k(a,b,c)> \frac{1}{D_k^2}\left(\alpha_k^{2c}+\alpha_k^{2c-2a}+\alpha_k^{2a}-6\right)-\frac{3}{D_k^3} (\alpha_k^{2c}+\alpha_k^{2c-2a}+\alpha_k^{2a}+4)
\\=\frac{1}{D_k^2}\alpha_k^{2c}\left[\left(1-\frac{3}{D_k}\right)\left(1+\alpha_k^{-2a}+\alpha_k^{2a-2c}\right)-\left(6+\frac{12}{D_k}\right)\alpha_k^{-2c}\right]\, .
\end{multline*}
Now, the factor $1-\frac{3}{D_k}=1-\frac{3}{\sqrt{k^2+4}}$ becomes positive for $k\ge 3$, so this time we need to apply the opposite Karamata bound \cite{K} (which becomes simply Jensen's inequality in this case)
$$\alpha_k^{-2a}+\alpha_k^{2a-2c}\ge 2\alpha_k^{\frac{-2a+2a-2c}{2}}=2\alpha_k^{-c},$$
yielding
$$m_k(a,b,c)>\frac{1}{D_k^2}\alpha_k^{2c}\left[\left(1-\frac{3}{D_k}\right)\left(1+2\alpha_k^{-c}\right)-\left(6+\frac{12}{D_k}\right)\alpha_k^{-2c}\right].$$
Let us consider the polynomial
$$p_k(x)=2\left( 1-\frac{3}{D_k}\right) x - \left(6+\frac{12}{D_k}\right)x^2.$$
Then, our bound can be written as
$$m_k(a,b,c)> \frac{1}{D_k^2}\alpha_k^{2c} \left[ 1-\frac{3}{D_k} + p_k(\alpha_k^{-c}) \right].$$
We know that $c=a+b\ge 2$, so $\alpha_k^{-c}\in (0,\alpha_k^{-2}]$, as $\alpha_k>1$, and therefore, $\lim_{c\to\infty} \alpha_k^{-c} = 0$. The polynomial $p_k(x)$ is a parabola with a negative leading coefficient, so its minimum in the interval $[0,\alpha_k^{-2}]$ is attained at one of the ends of the interval. A direct computation shows that $p_3(\alpha_3^{-2})<0=p_3(0)$, and hence
$$m_3(a,b,c)>\frac{1}{D_3^2}\alpha_3^{2c} \left[ 1-\frac{3}{D_3} + p_3(\alpha_3^{-2}) \right]= L_3\frac{1}{D_3^2}\alpha_3^{2c}.$$
On the other hand, for $k\ge 4$, we can prove that $p_k(\alpha_k^{-2})>0=p_k(0)$ as follows. The expression
$$\alpha_k^4p_k(\alpha_k^{-2})=2\alpha_k^2\left(1-\frac{3}{D_k}\right)-\left(6+\frac{12}{D_k}\right)$$
is clearly increasing in $k$, because $\alpha_k$ and $D_k$ are both increasing functions of $k$. A direct computation shows that for $k=4$ we have $\alpha_4^4p_4(\alpha_4^{-2})>0$, so $p_k(\alpha_k^{-2})$ must be positive for all $k\ge 4$. As a consequence,
$$m_k(a,b,c)>\frac{1}{D_k^2}\alpha_k^{2c} \left[ 1-\frac{3}{D_k} + p_k(\alpha_k^{-c}) \right] > \frac{1}{D_k^2}\alpha_k^{2c} \left[ 1-\frac{3}{D_k} + p_k(0) \right] = \frac{1}{D_k^2}\alpha_k^{2c} \left( 1-\frac{3}{D_k} \right)=L_k\frac{1}{D_k^2}\alpha_k^{2c}.$$
\end{proof}
We have the following lower bound for the constant $L_k$ in the lemma above.
\begin{lemma}
\label{lemma:Lkinequality}
For each $k\ge 2$, the constant $L_k$  satisfies
$$L_k>\alpha_k^{-2}.$$
\end{lemma}

\begin{proof}
For $k=2,3$, a direct computation shows that $\alpha_2^2L_2>1$ and $\alpha_3^2L_3>1$, so $L_k>\alpha_k^{-2}$ for $k=2,3$. For $k\ge 4$ we wish to prove that
$$L_k=1-\frac{3}{D_k}>\alpha_k^{-2}.$$
Rearranging the equation, this is equivalent to proving that for all $k\ge 4$
$$1>\frac{3}{D_k}+\alpha_k^{-2}=\frac{3}{\sqrt{k^2+4}}+\frac{4}{(k+\sqrt{k^2+4})^2}.$$
The right-hand side of this expression is decreasing in $k$ and for $k=4$ a direct computation shows that
$$\frac{3}{D_4}+\alpha_4^{-2} <1,$$
and hence the inequality holds for all $k\ge 4$.
\end{proof}

\begin{lemma}
\label{lemma:increasing}
Let $1\le a\le b\le c$ and $c\ge 3$. Suppose that $a\le a'\le c$ and $b\le b'\le c$. Then
$$m_k(a,b,c)\ge m_k(a',b',c)$$
and equality holds if and only if $a=a'$ and $b=b'$. In particular, if $(F_k(a),F_k(b),F_k(c))$ is an ordered minimal Markoff-Fibonacci $m$-triple, then
$$m_k(1,1,c)\ge m_k(a,b,c) \ge m_k(a,c-a-s,c),$$
where $s=1,$ for $k=2$ and $s=0,$ for $k\geq 3.$ 

\end{lemma}
\begin{proof}

The lemma and its proof are entirely analogous to Lemma 4.1 in \cite{ACMRS}, which addresses the case  \( k = 1 \). In this lemma, the starting point is \( a = 2 \) because \( F_1(2) = F_1(1) = 1 \). In our situation, with \( k \geq 2 \), the case \( a = 1 \) is also valid since \( F_k(2) > F_k(1) = 1 \).

\end{proof}

\begin{lemma}
\label{lemma:c=c'}
If 
 $(F_k(a),F_k(b),F_k(c))$ and $(F_k(a'),F_k(b'), F_k(c'))$ are two ordered minimal Markoff-Fibonacci $m$-triples with $c\ge c'$, then $c=c'.$
\end{lemma}

\begin{proof}
Assume that $m_k(a,b,c)=m=m_k(a',b',c')$.  By applying Lemma \ref{lemma:increasing} and Lemma \ref{lemma:KaramataBound}, it follows that
$$m=m_{2}(a,b,c)\ge m_{2}(a,c-a-1,c)> L_2 \frac{1}{D_2^2}\alpha_2^{2c}$$
if $k=2$ and
$$m=m_{k}(a,b,c)\ge m_{k}(a,c-a,c)> L_k \frac{1}{D_k^2}\alpha_k^{2c},$$
for any other $k\ge 3$. From Lemma \ref{lemma:Lkinequality} we know that $L_k>\alpha_k^{-2}$ for all $k\ge 2$, so
\begin{equation}
\label{eq:minimal1}
m_k(a,b,c)> L_{k}\frac{1}{D_k^2}\alpha_k^{2c}>\frac{1}{D_k^2}\alpha_k^{2c-2}\, .
\end{equation}
On the other hand, from Lemma \ref{lemma:increasing} we deduce that
\begin{multline}
\label{eq:minimal2}
m=m_k(a',b',c')\le m_k(1,1,c')=F_k(c')^2-3F_k(c')+2 <\\ \frac{1}{D_k^2}\alpha_k^{2c'}+\frac{1}{D_k^2}\bar{\alpha}_k^{2c'}+\frac{2}{D_k^2}(-1)^{c'}-1< \frac{1}{D_k^2} \alpha_k^{2c'}\, .
\end{multline}
Using equations \eqref{eq:minimal1} and \eqref{eq:minimal2} together, we obtain $\alpha_k^{2(c-1)} < D_k^2 m < \alpha_k^{2c'}$. Thus, $c'>c-1$. As we assumed $c'\le c$, we conclude that $c'=c$.
\end{proof}

\begin{lemma}
\label{lemma:aa'b'b}
Let $(F_k(a),F_k(b),F_k(c))$ and $(F_k(a'),F_k(b'),F_k(c))$ be two distinct ordered minimal Markoff-Fibonacci $m$-triples with the same third element. If $a\le a'$, then $a<a'\le b'<b$.
\end{lemma}

\begin{proof}
Suppose first that $a=a'$. Then, by  Lemma \ref{lemma:increasing}, the equality $m_k(a,b,c)=m_k(a',b',c')=m_k(a,b',c)$ is only possible if $b=b'$, in which case $(a,b,c)=(a',b',c')$, contradicting the assumption that the two $m$-triples are distinct. Thus $a<a'$. If $b\le b'$, then Lemma \ref{lemma:increasing} implies $m(a,b,c)<m(a',b',c)$, which is not possible as both are $m$-triples for the same $m$. Therefore, it follows that  $a<a'\le b'< b$.
\end{proof}

\begin{lemma}
\label{lemma:a+b=a'+b'} Let $(F_k(a),F_k(b),F_k(c))$ and $(F_k(a'),F_k(b'),F_k(c))$ be two ordered minimal Markoff-Fibonacci $m$-triples. Then $a+b=a'+b'$.
\end{lemma}

\begin{proof}
By Lemma \ref{lemma:aa'b'b} we can assume without loss of generality that $1\le a<a'\le b'<b\le c$. In particular, $b\ge 3$. Rearranging the equation $m_k(a,b,c)=m_k(a',b',c)$, yields
\begin{equation}
\label{eq:a+b=a'+b'eq1}
    F_k(a)^2+F_k(b)^2-F_k(a')^2-F_k(b')^2=3F_k(c)\left(F_k(a)F_k(b)-F_k(a')F_k(b')\right)\, .
\end{equation}
Since $b\ge 3$ and $a'\le b'<b$ we have
$$F_k(b)^2\geq k^2 F_k(b-1)^2> 2 F_k(b-1)^2\ge F_k(b')^2+F_k(a')^2,$$
so the left-hand side of equation \eqref{eq:a+b=a'+b'eq1} is always positive and, thus, so is the right-hand side. Let us see that this is impossible if $a'+b'>a+b$. Indeed,
\begin{equation*}
    \frac{F_k(a')F_k(b')}{F_k(a)F_k(b)} = \frac{(\alpha_k^{a'}-\bar{\alpha}_k^{a'})(\alpha_k^{b'}-\bar{\alpha}_k^{b'})}{(\alpha_k^a-\bar{\alpha}_k^{a})(\alpha_k^b-\bar{\alpha}_k^b)}\ge
    \frac{(\alpha_k^{a'}-\alpha_k^{-a'})(\alpha_k^{b'}-\alpha_k^{-b'})}{(\alpha_k^a+\alpha_k^{-a})(\alpha_k^b+\alpha_k^{-b})} = \frac{\alpha_k^{a'+b'}-\alpha_k^{b'-a'}-\alpha_k^{a'-b'}+\alpha_k^{-a'-b'}}{\alpha_k^{a+b}+\alpha_k^{b-a}+\alpha_k^{a-b}+\alpha_k^{-a-b}}\,.
\end{equation*}
 Assume that $a'+b'=a+b+r$ with $r>0$ and let $s=a+b$. Then $a'+b'=s+r$. Dividing the numerator and denominator by $\alpha_k^s$ yields
\begin{equation*}
\frac{\alpha_k^{a'+b'}-\alpha_k^{b'-a'}-\alpha_k^{a'-b'}+\alpha_k^{-a'-b'}}{\alpha_k^{a+b}+\alpha_k^{b-a}+\alpha_k^{a-b}+\alpha_k^{-a-b}} = \frac{\alpha_k^r-\alpha_k^{r-2a'}-\alpha_k^{r-2b'}+\alpha_k^{-2s-r}}{1+\alpha_k^{-2a}+\alpha_k^{-2b}+\alpha_k^{-2s}}=\alpha_k^r\frac{1-\alpha_k^{-2a'}-\alpha_k^{-2b'}+\alpha_k^{-2s-2r}}{1+\alpha_k^{-2a}+\alpha_k^{-2b}+\alpha_k^{-2s}}\, .
\end{equation*}
As $1\le a<a'\le b'<b$, we have $a\ge 1$, $a'\ge 2$, $b'\ge 2$, $b\ge 3$ and $s=a+b\ge 4$. Thus
$$\alpha_k^r\frac{1-\alpha_k^{-2a'}-\alpha_k^{-2b'}+\alpha_k^{-2s-2r}}{1+\alpha_k^{-2a}+\alpha_k^{-2b}+\alpha_k^{-2s}}\ge \alpha_k \frac{1-2\alpha_k^{-4}}{1+\alpha_k^{-2}+\alpha_k^{-6}+\alpha_k^{-8}}\ge 1.92 >1\,.$$
Therefore, $F_k(a')F_k(b')>F_k(a)F_k(b)$, which contradicts the positivity of both sides of equation \eqref{eq:a+b=a'+b'eq1}.

Therefore, we must have $a+b\ge a'+b'$. Suppose that $a'+b'=a+b-r$ with $r>0$ and let $s=a+b$ as before. Following the same logic as in the previous case, 
\begin{multline*}
    \frac{F_k(a')F_k(b')}{F_k(a)F_k(b)} = \frac{(\alpha_k^{a'}-\bar{\alpha}_k^{a'})(\alpha_k^{b'}-\bar{\alpha}_k^{b'})}{(\alpha_k^a-\bar{\alpha}_k^{a})(\alpha_k^b-\bar{\alpha}_k^b)}\le \frac{(\alpha_k^{a'}+\alpha_k^{-a'})(\alpha_k^{b'}+\alpha_k^{-b'})}{(\alpha_k^a-\alpha_k^{-a})(\alpha_k^b-\alpha_k^{-b})}
    = \frac{\alpha_k^{a'+b'}+\alpha_k^{b'-a'}+\alpha_k^{a'-b'}+\alpha_k^{-a'-b'}}{\alpha_k^{a+b}-\alpha_k^{b-a}-\alpha_k^{a-b}+\alpha_k^{-a-b}}\\=\alpha_k^{-r} \frac{1+\alpha_k^{-2a'}+\alpha_k^{-2b'}+\alpha_k^{-2s-2r}}{1-\alpha_k^{-2a}-\alpha_k^{-2b}+\alpha_k^{-2s}}
    \le \alpha_k^{-1}\frac{1+2\alpha_k^{-4}+\alpha_k^{-10}}{1-\alpha_k^{-2}-\alpha_k^{-6}}<0.53<\frac{8}{9}\, .
\end{multline*}

\noindent As a result,
$$1-\frac{F_k(a')F_k(b')}{F_k(a)F_k(b)}>1-\frac{8}{9}=\frac{1}{9}\ge \frac{1}{9F_k(a)^2}\, .$$
Multiplying both sides by $3F_k(a)F_k(b)F_k(c),$ results in
$$3F_k(c)\left(F_k(a)F_k(b)-F_k(a')F_k(b')\right)>\frac{F_k(c)F_k(b)}{3F_k(a)}\, .$$
Since $(F_k(a),F_k(b),F_k(c))$ is minimal, we have $F_k(c)\ge 3F_k(a)F_k(b)$. Consequently,
$$3F_k(c)\left(F_k(a)F_k(b)-F_k(a')F_k(b')\right)>\frac{F_k(c)F_k(b)}{3F_k(a)}\ge F_k(b)^2>F_k(b)^2-F_k(b')^2+F_k(a)^2-F_k(a')^2\, .$$
This contradicts equation \eqref{eq:a+b=a'+b'eq1}, and thus $a'+b'\geq a+b$ and therefore $a+b=a'+b'$.
\end{proof}

\begin{lemma}
\label{lemma:identityF3}
If $a$ is odd, $b$ is even, $b\geq a+3$ then
$$m_3(a,b,a+b)=m_3(a+1,b-1,a+b).$$
\end{lemma}
\begin{proof} Using Simson identity \eqref{eq:Simson} for $a$ odd,
\begin{multline*}
F_3(a)^2-F_3(a+1)^2=F_3(a)^2-F_3(a)F_3(a+2)+(-1)^{a+1}=\\
=F_3(a)(F_3(a)-F_3(a+2))+(-1)^{a+1}=-3F_3(a)F_3(a+1)+1\,.
\end{multline*}
Using a similar argument for $b$ even, we have 
$$F_3(b)^2-F_3(b-1)^2=3F_3(b)F_3(b-1)-1.$$
Adding both expressions yields
\begin{equation}
    \label{eq:identityF3_eq1}
    F_3(a)^2+F_3(b)^2-F_3(a+1)^2- F_3(b-1)^2=3(F_3(b)F_3(b-1)-F_3(a)F_3(a+1)).
\end{equation}
We obtain the following identities by applying Vajda's identity (see Lemma \ref{lemma:Vajda}) and considering that $a$ is odd and $b$ is even:
\begin{gather*}
    F_3(b)F_3(b-1)-F_3(a+b)F_3(b-a-1)=(-1)^{b-a-1}F_3(a)F_3(a+1)=F_3(a)F_3(a+1)\\
    F_3(a+1)F_3(b-1)-F_3(a)F_3(b)=(-1)^aF_3(1)F_3(b-a-1)=F_3(b-a-1)
\end{gather*}
Thus,
$$F_3(b)F_3(b-1)-F_3(a)F_3(a+1)=F_3(a+b)F_3(b-1-a)=F_3(a+b)(F_3(a+1)F_3(b-1)-F_3(a)F_3(b)).$$
Substituting back in \eqref{eq:identityF3_eq1} yields
$$F_3(a)^2+F_3(b)^2-F_3(a+1)^2-F_3(b-1)^2=3F_3(a+b)(F_3(a+1)F_3(b-1)-F_3(a)F_3(b)).$$
Rearranging this equation yields the required result.

\end{proof}

\begin{theorem}[Theorem \ref{th2} of the Introduction] 
If $m$ admits a minimal Markoff $m$-triple with $k$-Fibonacci components then it is unique except 
for $k=3$ and all  pairs of triples $(F_3(a),F_3(b),F_3(a+b)), \,\,(F_3(a+1),F_3(b-1),F_3(a+b)),$ for $a$ odd, $b$ even and $b\geq a+3.$
\end{theorem}

\begin{proof}
Let $(F_k(a),F_k(b),F_k(c))$ and $(F_k(a'),F_k(b'),F_k(c'))$ be a pair of ordered minimal $m$-triples contradicting the theorem. By Lemma \ref{lemma:c=c'}, it follows that $c=c'$. Moreover, by Lemma \ref{lemma:aa'b'b} we can assume without loss of generality that $1\le a<a'\le b'<b\le c$ and by Lemma \ref{lemma:a+b=a'+b'} we must have $a+b=a'+b'$.  Taking $n=a$, $i=b'-a$ and $j=b-b'=a'-a$ in Vajda's identity (Lemma \ref{lemma:Vajda}), we transform equation \eqref{eq:a+b=a'+b'eq1} into
\begin{multline}\label{aodd}
    F_k(a)^2+F_k(b)^2-F_k(a')^2-F_k(b')^2=3F_k(c)\left(F_k(a)F_k(b)-F_k(a')F_k(b')\right) \\= (-1)^{a+1} 3F_k(c)F_k(b'-a)F_k(b-b')\, .
\end{multline}

\noindent From the proof of Lemma \ref{lemma:a+b=a'+b'}, the left-hand side of this equality is positive, therefore $a$ is odd, and hence
\begin{equation}
    \label{eq:thmminimaleq1}
    F_k(a)^2+F_k(b)^2-F_k(a')^2-F_k(b')^2=3F_k(c) F_k(b'-a)F_k(b-b')\,.
\end{equation}

In the case $k=2,$ using (\ref{case_2}) from Lemma \ref{lemma:F(n+m)<=3F(n)F(m)} twice, we obtain that
$$F_2(b)\le 3F_2(b')F_2(b-b') \le 9F_2(a)F_2(b'-a)F_2(b-b')\, .$$
Multiplying by $F_2(b)$ and by minimality, $3F_2(a)F_2(b)\le F_2(c)$, it follows that 
$$F_2(b)^2\le 9F_2(a)F_2(b)F_2(b'-a)F_2(b-b')\le 3F_2(c)F_2(b'-a)F_2(b-b')$$
and as a consequence
$$F_2(b)^2-F_2(b')^2+F_2(a)^2-F_2(a')^2<F_2(b)^2\le 3F_2(c)F_2(b'-a)F_2(b-b'),$$
which contradicts equation \eqref{eq:thmminimaleq1}.

In the case $k\geq 4,$ suppose that $c=a+b$. We want to prove 
\begin{equation}\label{contra_c=a+b}F_k(b)^2-F_k(b')^2+F_k(a)^2-F_k(a')^2>3F_k(c)F_k(b'-a)F_k(b-b'),
\end{equation}
contradicting \eqref{eq:thmminimaleq1}. 
First, since $F_k(b)\geq k F_k(b-1)\geq 4 F_k(b')$ by equation \eqref{eq:basic_bound}, we have
\begin{equation}\label{parcial2}  F_k(a')^2+F_k(b')^2\leq 2 F_k(b')^2\le \frac{1}{8} F_k(b)^2 <\frac{F_k(b
)^2}{4}.
\end{equation}

\noindent Now, using equation \eqref{eq:bound_product} twice, it follows that
 $$3 F_k(a+b)F_k(b-b')F_k(b'-a)\leq 3F_k(a+b) F_k(b-a-1) \le 3 F_k(2b-2).$$
The inequality above and \eqref{parcial2} give 
$$F_k(a')^2+F_k(b')^2+3 F_k(a+b)F_k(b-b')F_k(b'-a)<\frac{F_k(b)^2}{4}+3 F_k(2b-2)$$
and by Lemma \ref{lemma:F(2b-2)lesssquares}
$$\frac{F_k(b)^2}{4}+3 F_k(2b-2)\le \frac{F_k(b)^2}{4}+\frac{3}{4}F_k(b)^2=F_k(b)^2.$$
Due to the two inequalities above, \eqref{contra_c=a+b} holds.

In the case $k=3,$ suppose that $c=a+b$ and $b'\leq b-2.$ We want to prove 
\begin{equation}\label{eq:case3}F_3(b)^2>F_3(a')^2+F_3(b')^2+3F_3(a+b)F_3(b'-a)F_3(b-b'),
\end{equation} which contradicts equation \eqref{eq:thmminimaleq1}. Repeating the argument above, $$3F_3(a+b)F_3(b'-a)F_3(b-b')\le 3F_3(2b-2) \le \frac{3}{4}F_3(b)^2.$$
\noindent On the other hand, if $a'\leq b'\leq b-2,$ since $F_3(b)\geq 9F_3(b-2),$ we have
$$F_3(a')^2+F_3(b')^2\leq 2F_3(b')^2\leq 2F_3(b-2)^2\le \frac{2}{9} F_3(b)^2 <\frac{1}{4}F_3(b)^2. $$
Adding the two inequalities above, \eqref{eq:case3} holds.

In the case $k\geq 3,$ we first consider $c\geq a+b+1$. We will show that 
\begin{equation}F_k(b)^2-F_k(b')^2+F_k(a)^2-F_k(a')^2<3F_k(c)F_k(b'-a)F_k(b-b'),
\end{equation} which contradicts equation \eqref{eq:thmminimaleq1}. Then, since $F_k(b')> F_k(a)$ it is enough to show that
\begin{equation} \label{contra_c>a+b}F_k(b)^2<3F_k(a+b+1)F_k(b'-a)F_k(b-b').
\end{equation}
\noindent By using equation \eqref{eq:bound_product2} twice, we obtain
\begin{multline*}3F_k(a+b+1)F_k(b'-a)F_k(b-b')\ge 3F_k(a+b+1)\frac{1}{(1+\frac{1}{9})}F_k(b-a-1)\ge \\\frac{3}{\left(1+\frac{1}{9}\right)^2}F_k(2b-1)>F_k(2b-1).
\end{multline*}
On the other hand, applying  formula \eqref{eq:sum} to $b-1$ and $b,$ it follows that
$$F_k(2b-1)=F_k(b)^2+F_k(b-1)^2>F_k(b)^2.$$
The two inequalities above show that \eqref{contra_c>a+b} holds.

Finally, we study the last case; $k=3,$  $c=a+b$, $b'= b-1$ and $a$ odd (see equation \eqref{aodd}). This is precisely addressed in Lemma \ref{lemma:identityF3},  which identifies the minimal pairs of  Markoff $m$-triples with $k$-Fibonacci components satisfying $m=m_3(a,b,a+b)=m_3(a+1,b-1,a+b),$ where $b$ is even. Note that the condition $b\geq a+3$ in that lemma implies that the triple $(F_3(a+1),F_3(b-1),F_3(a+b))$ is ordered, so $(F_3(a+1),F_3(b-1),F_3(a+b))$ and $(F_3(a),F_3(b),F_3(a+b))$ are distinct. This, however, does not hold if $b=a+1.$
If $b$ were odd, we would have in the last equality of Lemma \ref{lemma:a+b=a'+b'}
    $$F_3(a)^2+F_3(b)^2-F_3(a+1)^2-F_3(b-1)^2=3F_3(a+b)(F_3(a+1)F_3(b-1)-F_3(a)F_3(b))+6.$$
Therefore, if $b$ were odd, $m_3(a,b,a+b)>m_3(a+1,b-1,a+b).$

\end{proof}

%

\end{document}